\documentclass{amsart}
\usepackage[fontsize=10pt]{scrextend}
\usepackage{nicefrac}
\usepackage{enumitem}
\usepackage{graphicx}
\usepackage{mathrsfs}
\usepackage{xcolor}

\newtheorem{theorem}{Theorem}[section]

\newtheorem{lemma}[theorem]{Lemma}

\theoremstyle{definition}

\theoremstyle{remark}

\newtheorem{question}[theorem]{Question}

\numberwithin{equation}{section}

\begin{document}

\title{On the set of supercyclic operators}

\author{Thiago R. Alves}
\address{Departamento de Matem\'{a}tica,
	Instituto de Ci\^{e}ncias Exatas,
	Universidade Federal do Amazonas,
	69.077-000 -- Manaus -- Brazil}
\email{alves@ufam.edu.br}
\thanks{The first author was financed in part by the Coordena\c{c}\~{a}o de Aperfei\c{c}oamento de Pessoal de N\'{\i}vel Superior – Brasil (CAPES) – Finance Code 001 and FAPEAM}

\author{Gustavo C. Souza}
\address{Departamento de Matem\'{a}tica,
	Instituto de Ci\^{e}ncias Exatas,
	Universidade Federal do Amazonas,
	69.077-000 -- Manaus -- Brazil}
\email{gustavo.souza@super.ufam.edu.br }
\thanks{}


\date{}

\dedicatory{}


\begin{abstract}
In this article, we address a problem posed by F. Bayart regarding the existence of an infinite-dimensional closed vector subspace (excluding the null operator) within the set of supercyclic operators on Banach spaces. We resolve this problem by establishing the existence of the closed subspace. Furthermore, we prove that the set of supercyclic operators on $\ell_1$ contains, up to the null operator, an isometric copy of $\ell_1$.
\end{abstract}

\maketitle

\section{Introduction and main result}

A {\it supercyclic operator} $T$ from a Banach space $X$ to itself is defined as a linear continuous  operator for which there exists a vector $x$ in $X$ such that the set 
\begin{align} \label{set-dens-super}
\{\lambda \, T^k(x) : \lambda \in \mathbb K, \, k \geq 0\}
\end{align}
is dense in $X$, where $\mathbb K$ is either the real field $\mathbb R$ or the complex field $\mathbb C$, depending on whether $X$ is a real or complex vector space, respectively. The term ``supercyclic'' was first used by H. M. Hilden and L. J. Wallen in \cite[Sect. 5]{HilWal74}. Supercyclic operators have attracted significant attention in the field of linear dynamics, and extensive research has been dedicated to them, as evidenced by the existing literature and references (cf. \cite[Chs. 1 and 9]{BayMath-book}, its citing articles, and its cited references).

This article focuses on a question originally raised by F. Bayart \cite[p. 302, Remark (2)]{Bayart}; also referred to as Problem 47 in \cite{GuiMonZiz}. The question pertains to the existence of an infinite-dimensional closed vector subspace, excluding the null operator, within the set of supercyclic operators. While F. Bayart has confirmed the affirmative answer in the context of operators defined on infinite-dimensional complex separable Hilbert spaces, the question has remained unresolved in the broader framework of separable Banach spaces, including real separable Hilbert spaces.

In this article, we tackle the question posed by F. Bayart in the context of operators defined on infinite-dimensional separable Banach spaces, covering both real and complex cases. Specifically, we offer a solution by establishing the existence of an infinite-dimensional closed vector subspace, excluding the null operator, within the set of supercyclic operators. Furthermore, we prove that the set of supercyclic operators on $\ell_1$ contains (up to the null operator) an isometric copy of $\ell_1$.

Similar investigations, aimed at uncovering linear structures within nonlinear sets, have been conducted across various domains of Analysis in the past two decades, as evidenced by the comprehensive survey article \cite{Ber-GonPelSeo-Sep} and its citations. We adopt the established terminology introduced in \cite{AronGurSeo}, whereby a subset of a topological vector space is considered {\it``spaceable''} if it contains (up to the null vector) an infinite-dimensional closed vector subspace.

Hence, considering a separable Banach space $X$ and denoting $\mathcal{L}(X)$ as the set of all continuous linear operators on $X$, our main result can be stated as follows:

\begin{theorem} \label{thm1}
	The set of supercyclic operators in $\mathcal{L}(X)$ is spaceable. Furthermore, the set of supercyclic operators in $\mathcal{L}(\ell_1)$ contains (up to the null operator) an isometric copy of $\ell_1$.
\end{theorem}

It is worth noting that D. Kitson and R. M. Timoney provided a general condition in \cite[Theorem 3.3]{KitTim11} for proving the spaceability of a set $A$ in a Fréchet space $Z$. To the best of our knowledge, this is the only known general condition in Fréchet space for establishing spaceability. Unfortunately, the standard method of applying this condition necessitates the complement, denoted as $Z \setminus A$, of the set to be a vector space in its own right. Regrettably, this requirement is not met in our specific scenario. To illustrate this point, let us examine the following linear continuous operators:
\small
 $$(\lambda_k) \in \ell_1 \mapsto (0,\lambda_3,\lambda_4,\lambda_5, \ldots) \in \ell_1 \ \ \mbox{and} \ \ (\lambda_k) \in \ell_1 \mapsto (\lambda_2, 0, 0, 0, \ldots) \in \ell_1.$$ \normalsize
Neither of these operators is supercyclic on $\ell_1$, which is evident. However, their sum yields the backward shift operator, a well-known example of a supercyclic operator (cf. \cite[Example 1.15]{BayMath-book} for more details). Consequently, we pursue a more constructive approach to prove Theorem \ref{thm1}.

The remaining sections of the article is structured as follows. In Section \ref{back-notations}, we present the necessary notations used throughout the article, along with some background information relevant to the topic. Section \ref{main-thm} contains the proof of the main theorem, which is presented in a concise manner.  The proof relies on the use of Lemma \ref{lem2}, which was proved by employing the Supercyclicity Criterion, as discussed in Section \ref{Sec-lemmas}. Section \ref{Sec-lemmas} is dedicated to the technical aspects of the proof, where we provide four technical lemmas leading up to Lemma \ref{lem2}, which is directly utilized in the proof of the main result.

\section{Background and notations} \label{back-notations}

Throughout, $X$ denotes an infinite-dimensional separable Banach space, which can be either real or complex. The symbol $\mathcal{L}(X)$ stands for the set of all continuous linear operators from the normed space $X$ to itself. The notation $U^k$ signifies the composition of the map $U$ with itself $k$ times, where $U$ is a given map.

As a consequence of \cite[Theorem 1]{Ovsepian-Pelcz}, we can and will fix a biorthogonal system $(x_n, x_n^*)_{n=1}^\infty \subset X \times X^*$ with the following properties:
	\begin{align}
	&\sup_{n \in \mathbb N}\|x^*_n\| =: C_X < \infty \mbox{ and }\|x_n\| = 1 \mbox{ for each } n; \label{eq-2.1}\\ &x_m^*(x_n) = \delta_{mn}, \mbox{ where } \delta_{mn} \mbox{ denotes the Kronecker delta}; \label{eq-2.2}\\ &\text{span}\{x_n : n \in \mathbb N\}  \mbox{ is dense in } X; \label{eq-2.3}\\ &x = 0 \mbox{ if } x_n^*(x) = 0 \mbox{ for each } n. \label{eq-2.4}
\end{align}

For what will follow, it is worth noting that from $\eqref{eq-2.1}$ and $\eqref{eq-2.2}$, we have $C_X \geq 1$ (and thus $C_X^{-1} \leq 1$).

For each sequence ${\bf w} = (w_n) \in \ell_1$ of positive numbers, we can define a weighted backward shift operator $B_{\bf w} : X \to X$ as follows:
\begin{align*}
	\begin{split}
	B_{\bf w}(x) := \sum_{n=1}^\infty x^*_{n+1}(x) \, w_n \, x_n.
	\end{split}
\end{align*}
We have that $B_{\bf w}$ is well-defined, linear, and continuous. Moreover, ${\|B_{\bf w}\|}_{\mathcal{L}(X)} \leq 1$ whenever ${\|{\bf w}\|}_1 \leq C_X^{-1}$. Indeed, we can prove the well-definedness  of $B_{\bf w}$ through the following calculation:
\begin{align} \label{conta-1}
	\begin{split}
		\sum_{n=1}^\infty {\|x^*_{n+1}(x) \, w_n \, x_n\|}_X \stackrel{\eqref{eq-2.1}}{\leq} \sup_{i \geq 2}|x^*_i(x)| \ \sum_{n=1}^\infty w_n \stackrel{\eqref{eq-2.1}}{\leq} C_X\, {\|x\|}_X \, {\|{\bf w}\|}_1.
	\end{split}
\end{align}
The linearity of $B_{\bf w}$ is evident, while $(\ref{conta-1})$ shows that its operator norm is bounded by $1$ when ${\|{\bf w }\|}_1 \leq C_X^{-1}$. Throughout this article, we make the overarching assumption that
\begin{align} \label{sequence-w}
{\bf w} = (w_n) \in \ell_1, \ \ w_n \searrow  0 \ \ \mbox{and} \ \ {\|{\bf w}\|}_1 \leq C_X^{-1} \leq 1.
\end{align}
Note that the second assumption explicitly designates $(w_n)$ as a decreasing sequence of positive real numbers.

\vspace{0.1in}

We also use the following notation throughout Section \ref{Sec-lemmas}. For a vector $x \in X$, its support is denoted by $supp(x) = \{n : x^*_n(x) \neq 0\}$. Moreover, for any pair of natural numbers $M \leq N$, we set $X_{[M,N]} := span\{x_n : M \leq n \leq N\}$. To simplify notation, we set $X_N := X_{[1,N]}$ for every $N \in \mathbb N := \{1,2,\ldots\}$. Finally, we define $X_\infty$ as the union of all $X_{[M,N]}$; that is, $X_\infty$ is the vector subspace of $X$ generated by $(x_n)$, and it follows from \eqref{eq-2.3} that $X_\infty$ is dense in $X$.

For any $\lambda = (\lambda_k)_{k=1}^\infty \in \ell_1 \setminus \{0\}$, we define $\mathfrak{p}_\lambda$ as the smallest element in $supp(\lambda) := \{k : \lambda_k \neq 0\}$. Also, for any $y \in X_\infty \setminus \{0\}$, $\mathfrak{q}_y$ is defined as the largest element in $supp(y)$, which is well-defined because of \eqref{eq-2.2} and \eqref{eq-2.4}. In addition to that, we consider a function  $F_\lambda : \mathbb N \to \mathbb (0,\infty)$ given by
	\begin{align*}
		F_{\lambda}(d) := \dfrac{C_X  (d+1)!\left(\displaystyle\max_{0\leq i < d} |\lambda_{\mathfrak{p}_\lambda + i}|\right)^{d-1}}{{|\lambda_{\mathfrak{p}_\lambda}|}^d \, w_{\mathfrak{p}_\lambda + d}^{d \cdot \mathfrak{p}_\lambda}}.
	\end{align*}
Note that $F_\lambda$ is an increasing function. Indeed, since $\displaystyle\max_{0 \leq i < d}|\lambda_{\mathfrak{p}_\lambda+i}| \geq |\lambda_{\mathfrak{p}_\lambda}| > 0$ for each $d \in \mathbb N$ and $0 < w_{k+1} < w_k < C_X^{-1} \leq 1$ for each $k \in \mathbb N$, we have
\begin{align*}
	\begin{split}
 \left(\dfrac{\displaystyle\max_{0 \leq i < d} |\lambda_{\mathfrak{p}_\lambda +i}|}{|\lambda_{\mathfrak{p}_\lambda}| \, w_{\mathfrak{p}_\lambda + d}^{\mathfrak{p}_\lambda}}\right)^{d-1} \dfrac{1}{|\lambda_{\mathfrak{p}_\lambda}| \, w_{\mathfrak{p}_\lambda + d}^{\mathfrak{p}_\lambda}} \leq \left(\dfrac{\displaystyle\max_{0 \leq i < d+1} |\lambda_{\mathfrak{p}_\lambda +i}|}{|\lambda_{\mathfrak{p}_\lambda}| \, w_{\mathfrak{p}_\lambda + d + 1}^{\mathfrak{p}_\lambda}}\right)^{d} \dfrac{1}{|\lambda_{\mathfrak{p}_\lambda}| \, w_{\mathfrak{p}_\lambda + d + 1}^{\mathfrak{p}_\lambda}},
\end{split}
\end{align*}
which promptly  implies $F_\lambda(d) \leq F_\lambda(d+1)$.

Now, given a family of sequences $(\lambda_k^m)_{k=1}^\infty \subset \mathbb K$, $m \in \mathbb{N}$, such that $(\lambda_k^m)_{m=1}^\infty$ converges for each $k$ and $\displaystyle\inf_{m \in \mathbb N} |\lambda_{k_0}^m| =: \delta > 0$ for some $k_0 \in \mathbb N$, we define a function $G_{k_0,\delta}: \mathbb{N} \to (0,\infty)$ as follows:
\begin{align} \label{def-G}
	G_{k_0,\delta}(d) := \dfrac{ C_X(d+1)! \left(\displaystyle\sup_{j \in \mathbb N} \max_{0\leq i < d} |\lambda_{k_0 + i}^j|\right)^{d-1}}{\delta^d \, w_{k_0+d}^{d \cdot k_0}}.
\end{align}
We can similarly prove that $G_{k_0,\delta}$ is an increasing function. This is due to the fact that $\displaystyle\sup_{j \in \mathbb N} \max_{0\leq i < d} |\lambda_{k_0 + i}^j| \geq \delta > 0$ for each $d \in \mathbb N$ and $0 < w_{k+1} < w_k < C_X^{-1} \leq 1$ for each $k \in \mathbb N$.

\section{Proof of the main result} \label{main-thm}

In accordance with the notations introduced in Section \ref{back-notations}, we define the operator $T : \ell_1 \to \mathcal{L}(X)$ as follows:
\begin{align} \label{eq.1}
	T_\lambda := T(\lambda) = \sum_{k=1}^\infty \lambda_k \, B_{\bf w}^k \ \ \mbox{with} \ \ \lambda = (\lambda_k)_k \in \ell_1.
\end{align}

For every $x$ in the unit ball of $X$ and each $\lambda = (\lambda_k) \in \ell_1$, we have
$$\sum_{k=1}^\infty {\|\lambda_k \, B_{\bf w}^k(x)\|}_X \leq \sum_{k=1}^\infty |\lambda_k| < \infty,$$
where the first inequality follows from ${\|B_{\bf w}\|}_{\mathcal{L}(X)} \leq 1$ (cf. Section \ref{back-notations}). Thus, $T$ is clearly well-defined, linear, continuous, and ${\|T(\lambda)\|}_{\mathcal{L}(X)} \leq {\|\lambda\|}_1$ for each $\lambda \in \ell_1$. Furthermore, $T$ is injective since $T(\lambda) = 0$ implies 
$$\sum_{n=1}^N \left(\lambda_{N+1-n} \cdot \prod_{i=n}^N w_i \right) x_n = \sum_{k=1}^\infty \lambda_k \, B_{\bf w}^k(x_{N+1}) = 0,$$ 
for all $N \in \mathbb N$, and thus $\lambda = 0$. Finally, by utilizing Lemma \ref{lem2}, we conclude that the closure of the range of $T$ contains only supercyclic operators (up to the null operator). This completes the proof of the first statement of the theorem.

For $X = \ell_1$, we can consider $T$ defined by using $B$, the standard backward shift operator on $\ell_1$, instead of $B_{\bf w}$. In this case, all the arguments mentioned earlier apply seamlessly. In particular, it is worth noting that the validity of Lemma \ref{lem2} follows by setting $w_k = 1$ in any part of Section \ref{Sec-lemmas} and letting $X = \ell_1$ be equipped with its canonical Schauder basis $(e_n)$. This basis gives rise to a biorthogonal system $(e_n, e_n^*) \subset \ell_1 \times \ell_\infty$, satisfying conditions $\eqref{eq-2.1}$-$\eqref{eq-2.4}$.

 Moreover, for this specific case where $X = \ell_1$, we can establish that $T$ is an isometry by demonstrating the reverse inequality as follows:
\small
\begin{eqnarray*}
	\begin{split}
		\|T(\lambda)\|_{\mathcal{L}(\ell_1)} \geq \sup_{n \in \mathbb N} \left\|\sum_{k=1}^\infty \lambda_k B^k(e_{n+1})\right\|_1 = \sup_{n \in \mathbb N} {\|(\lambda_n, \lambda_{n-1}, \ldots, \lambda_1, 0, 0, \ldots)\|}_1 = {\|\lambda\|}_1.
	\end{split}
\end{eqnarray*}\normalsize
This concludes the proof of Theorem \ref{thm1}.

\vspace{0.1in}

In light of Theorem \ref{thm1}, it is natural to raise the following question.

\begin{question}
	Which spaces $X$, other than $\ell_1$, permit the existence of (isometric) copies of classical infinite-dimensional Banach spaces,  excluding the null operator, within the set of supercyclic operators in $\mathcal{L}(X)$?
\end{question}

\section{Technical lemmas} \label{Sec-lemmas}

In this section, we present four technical lemmas, with the last one directly applied in the proof of Theorem \ref{thm1}. For any notation that has not been introduced within the proofs or statements of the lemmas, please refer to Section 2 for its definition and explanation.

\begin{lemma} \label{lem00}
	Let $T$ be defined as in $(\ref{eq.1})$, and let $0 \neq \lambda = (\lambda_k)_{k=1}^\infty \in \ell_1$. Then, there exists a family of linear maps $S_{\lambda,d}: X_d \to X_{[\mathfrak{p}_\lambda + 1, \mathfrak{p}_\lambda+d]}$, $d \in \mathbb N$, satisfying the following properties for each $y \in X_d$:
	\begin{enumerate}[label=$(\mathsf{A1})$]
		\item \label{a1} $(T_\lambda \circ S_{\lambda,d})(y) = y$.
	\end{enumerate}
	\begin{enumerate}[label=$(\mathsf{B1})$]
		\item \label{b1} ${\|S_{\lambda,d}(y)\|}_X \leq F_\lambda(d) \cdot {\|y\|}_X$.
	\end{enumerate}
	\begin{enumerate}[label=$(\mathsf{C1})$]
	\item \label{c1} $\mathfrak{q}_{S_{\lambda,d}(y)} = \mathfrak{p}_\lambda + \mathfrak{q}_y$ if $y \not= 0$.
\end{enumerate}
\end{lemma}
\begin{proof}
	To define $S_{\lambda,d}$, we will use the following notation:
	$$S_{\lambda,d}(y) :=: \sum_{n=\mathfrak{p}_\lambda+1}^{\mathfrak{p}_\lambda+d} b_n^y \, x_n \ \  \mbox{with}  \ \ y = \sum_{n=1}^d a_n \, x_n \in X_d.$$
	Next, the goal is to establish, for each $y \in X_d$, the unique existence of $b_n^y \in \mathbb K$, $n=\mathfrak{p}_\lambda+1,\mathfrak{p}_\lambda+2,\ldots,\mathfrak{p}_\lambda+d$, for which  $(\mathsf{A1})$ holds. By setting $b_n^y$ to be $0$ if $n < \mathfrak{p}_\lambda +1 $ or $n > \mathfrak{p}_\lambda + d$, we can express the left side of \ref{a1} as follows:
	\small
	\begin{align}
		\begin{split} \label{leftsidesystem}
			(T_\lambda \circ S_{\lambda,d})(y) &= \sum_{k=1}^\infty \lambda_k \, \sum_{m=1}^\infty b_m^y \,B_{\bf w}^k(x_m) = \sum_{k=1}^\infty \sum_{m=k+1}^\infty \lambda_k \, b_m^y \left(\prod_{i=m-k}^{m-1} w_i \right) x_{m-k} \\ &\stackrel{(n := m-k)}{=} \sum_{n=1 }^\infty \left(\sum_{k = 1}^\infty \lambda_k \, b_{k+n}^y \, \prod_{i = n}^{n+k-1} w_i \right) x_n \\ &= \sum_{n=1}^d \left(\sum_{k = \mathfrak{p}_\lambda}^{\mathfrak{p}_\lambda+d-n} \lambda_k \, b_{k+n}^y \, \prod_{i = n}^{n+k-1} w_i \right) x_n,
		\end{split}
	\end{align}\normalsize
	where the last equality follows from both $b_{k+n}^y = 0$ if $k+n > \mathfrak{p}_\lambda + d$ and $\lambda_k = 0$ if $k < \mathfrak{p}_\lambda$.
	
Hence, by examining \eqref{leftsidesystem}, we can deduce that proving the unique existence of the values $b_n^y \in \mathbb K$, $n=\mathfrak{p}_\lambda+1,\mathfrak{p}_\lambda+2,\ldots,\mathfrak{p}_\lambda+d$, satisfying \ref{a1} is equivalent to demonstrating the solvability of the linear system represented by
	\begin{eqnarray} \label{system1}
		\begin{bmatrix}
			{\bf M}[1,1] &  {\bf M}[1,2] & {\bf M}[1,3] & \cdots & {\bf M}[1,d] \\
			0 & {\bf M}[2,2] & {\bf M}[2,3] & \cdots & {\bf M}[2,d] \\
			0 & 0 & {\bf M}[3,3] & \cdots & {\bf M}[3,d]\\
			\vdots & \vdots & \vdots & \ddots & \vdots \\
			0 & 0 & 0 & \cdots & {\bf M}[d,d]
		\end{bmatrix}
		\begin{bmatrix}
			b_{\mathfrak{p}_\lambda+1}^y \\
			b_{\mathfrak{p}_\lambda+2}^y \\
			b_{\mathfrak{p}_\lambda+3}^y \\
			\vdots \\
			b_{\mathfrak{p}_\lambda+d}^y
		\end{bmatrix}
		=
		\begin{bmatrix}
			a_1 \\
			a_2 \\
			a_3 \\
			\vdots \\
			a_d
		\end{bmatrix},
	\end{eqnarray}
where
\begin{align} \label{coef-M[n,k]}
{\bf M}[n,k] := \lambda_{\mathfrak{p}_\lambda+k-n} \prod_{i=n}^{\mathfrak{p}_\lambda+k-1} w_i, \mbox{ with } n \leq k \mbox{ and } n,k=1, 2, \ldots, d.
\end{align}

Denoting by ${\bf M}$ the matrix of coefficients in $(\ref{system1})$, we can see that the determinant of ${\bf M}$ is $\prod_{i=1}^d {\bf M}[i,i] \not= 0$, implying that the above linear system has a unique solution. Thus, we can define $S_{\lambda,d}$ as follows:
$$S_{\lambda,d}\left(\sum_{n=1}^d a_n \, x_n\right) = \sum_{n=\mathfrak{p}_\lambda+1}^{\mathfrak{p}_\lambda+d} b_{n}^y \, x_{n},$$
where $(b_{\mathfrak{p}_\lambda+1}^y, b_{\mathfrak{p}_\lambda+2}^y, \ldots, b_{\mathfrak{p}_\lambda+d}^y) \in \mathbb{K}^d$ is the unique solution for the system (\ref{system1}). We note that ${\bf M}^{-1}$ is the matrix representation of $S_{\lambda,d}$ with respect to the bases $\{x_1,x_2,\ldots,x_d\}$ and $\{x_{\mathfrak{p}_\lambda + 1},x_{\mathfrak{p}_\lambda + 2},\ldots, x_{\mathfrak{p}_\lambda + d}\}$.

To verify the validity of condition \ref{b1}, we recall (cf. \cite[Ch. VI, Sect. 8]{Lang87-book}) that the coefficients of the inverse matrix of ${\bf M}$, denoted as ${({\bf M}^{-1})}[n,k]$, are given by
\begin{align} \label{coef-inv-M}
{({\bf M}^{-1})}[n,k] = \dfrac{(-1)^{n+k} \, Det(\widehat{\textbf{M}}_{kn})}{Det({\bf M})}.
\end{align}
Here, $\widehat{\textbf{M}}_{kn}$ represents the matrix obtained by eliminating the $k$th row and the $n$th column from ${\bf M}$. Moreover, it is pertinent to recall (cf. \cite[Ch. VI, Sect. 7]{Lang87-book}) that when $d>1$, the determinant of $\widehat{\textbf{M}}_{kn}$ can be computed using the expression
\begin{align} \label{det-of-Mij}
	Det(\widehat{\textbf{M}}_{kn}) = \sum_{\sigma \in \mathscr{P}_{d-1}} \operatorname{sgn}(\sigma) \prod_{\ell=1}^{d-1} \widehat{\textbf{M}}_{kn}[\sigma(\ell),\ell],
\end{align}
where $\mathscr{P}_{d-1}$ stands for the symmetric group of $d-1$ elements and $\operatorname{sgn}(\sigma)$ the signature of the permutation $\sigma$. On the other hand, the determinant of ${\bf M}$ is given by
\begin{align} \label{det-of-M}
Det({\bf M})  = \prod_{j=1}^d {\bf M}[j,j] = \lambda_{\mathfrak{p}_\lambda}^d \, \prod_{j=1}^d  \, \prod_{i=j}^{\mathfrak{p}_\lambda +j-1} w_i.
\end{align}
Thus, we can conclude that
\begin{align} \label{domi-coef-inv-M}
	\begin{split}
	|{\bf M}^{-1}[n,k]| &\stackrel{\eqref{coef-inv-M}+\eqref{det-of-M}}{=} \dfrac{|Det(\widehat{{\bf M}}_{kn})|}{|\lambda_{\mathfrak{p}_\lambda}^d| \prod_{j=1}^d \prod_{i=j}^{\mathfrak{p}_\lambda + j - 1} w_i} \\ &\stackrel{\eqref{det-of-Mij}}{\leq} \dfrac{(d-1)!  \displaystyle\max_{1\leq i \leq j \leq d}|{\bf M}[i,j]|^{d-1}}{|\lambda_{\mathfrak{p}_\lambda}^d| \prod_{j=1}^d \prod_{i=j}^{\mathfrak{p}_\lambda + j - 1} w_i} \\ &\stackrel{\eqref{sequence-w} + \eqref{coef-M[n,k]}}{\leq} \dfrac{(d-1)!  \displaystyle\max_{1\leq i \leq j \leq d} |\lambda_{\mathfrak{p}_\lambda +j-i}|^{d-1}}{|\lambda_{\mathfrak{p}_\lambda}^d| \, w_{\mathfrak{p}_\lambda + d}^{d \cdot \mathfrak{p}_\lambda}}.
\end{split}
\end{align}

Therefore, considering a given $y = \sum_{n=1}^d a_n \, x_n \in X_d$, and utilizing the fact that ${\bf M}^{-1}$ is the matrix representation of $S_{\lambda,d}$ in terms of the bases $\{x_1,x_2,\ldots,x_d\}$ and $\{x_{\mathfrak{p}_\lambda + 1},x_{\mathfrak{p}_\lambda + 2},\ldots, x_{\mathfrak{p}_\lambda + d}\}$, we obtain
\begin{align*}
	{\|S_{\lambda,d}(y)\|}_X &\stackrel{\eqref{eq-2.2}}{=} {\left\|\sum_{n=1}^d\sum_{k=n}^d {\bf M}^{-1}[n,k] \, x_k^*(y) \, x_{\mathfrak{p}_\lambda + n} \right\|}_X \stackrel{\eqref{eq-2.1}}{\leq} C_X \, \sum_{n=1}^d\sum_{k=n}^d |{\bf M}^{-1}[n,k]| \, {\|y\|}_X \\ &\stackrel{\eqref{domi-coef-inv-M}}{\leq} \dfrac{C_X \, d \, (d+1) \, (d-1)!  \displaystyle\max_{1\leq i \leq j \leq d} |\lambda_{\mathfrak{p}_\lambda +j-i}|^{d-1}}{|\lambda_{\mathfrak{p}_\lambda}^d| \, w_{\mathfrak{p}_\lambda + d}^{d \cdot \mathfrak{p}_\lambda}} {\|y\|}_X,
\end{align*}
as stated in \ref{b1}.

To establish \ref{c1}, consider $y = \sum_{n=1}^{d} a_n \, x_n \neq 0$. Let $(a_1,a_2, \ldots, a_{\mathfrak{q}_y}, 0, \ldots, 0) \in \mathbb K^d$. From \eqref{system1}, exploiting the upper triangularity of matrix {\bf M} and the non-zero entries ${\bf M}[i,i]$ for all $i=1,2,\ldots,d$, we deduce that $b_{\mathfrak{p}_\lambda + \mathfrak{q}_y}^y \neq 0$ and $b_n^y = 0$ for $n > \mathfrak{p}_\lambda + \mathfrak{q}_y$. Thus, we have successfully established the validity of \ref{c1}.

\end{proof}

\begin{lemma} \label{lem1}
	Let  $T$ be as $(\ref{eq.1})$ and let $0 \not= \lambda = (\lambda_k)_{k=1}^\infty \in \ell_1$. Then there exists a map $S_\lambda : X_{\infty} \to X$ such that $S_\lambda(X_\infty) \subset X_\infty$ and, for all $k \in \mathbb N$ and $y \in X_\infty$, we have:	
	\begin{enumerate}[label=$(\mathsf{A2})$]
		\item \label{a2} $(T_\lambda^k \circ S_\lambda^k)(y) = y$.
	\end{enumerate}
	\begin{enumerate}[label=$(\mathsf{B2})$]			
		\item \label{b2}$ {\|S_\lambda^k(y)\|}_X \leq \left[F_\lambda(d_k^y)\right]^k \cdot {\|y\|}_X$ if $y \not= 0$, where $d_k^y := (k-1) \, \mathfrak{p}_\lambda + \mathfrak{q}_{y}$.
	\end{enumerate}
\end{lemma}
\begin{proof}	
	Choose a family of linear maps $S_{\lambda,d} : X_d \to X_{[\mathfrak{p}_\lambda+1, \mathfrak{p}_\lambda+d]}$, $d \in \mathbb N$, as in Lemma \ref{lem00}. Let us now define $S_\lambda : X_\infty \to X$ as follows. For $y=0$, we set $S_\lambda(0) = 0$. For $0 \neq y \in X_\infty$, we define 	 
	\begin{align} \label{def.of.Sz}
		S_\lambda(y) := S_{\lambda,\mathfrak{q}_y}(y).
	\end{align}
It can be easily verified that $S_\lambda$ is well-defined and $S_\lambda(X_\infty) \subset X_\infty$. Furthermore, for every $y \in X_\infty \setminus \{0\}$ and $k \in \mathbb N$, we can deduce from \eqref{def.of.Sz} and Lemma \ref{lem00}\ref{c1} that
	\begin{align} \label{composta-Sk}
		S_\lambda^k(y) = (S_{\lambda,d_k^y} \circ S_{\lambda,d_{k-1}^y} \circ \cdots \circ  S_{\lambda,d_{2}^y} \circ S_{\lambda,d_1^y})(y),
	\end{align}
where $d_{\ell}^y = (\ell-1) \, \mathfrak{p}_\lambda + \mathfrak{q}_{y}$, $\ell = 1,2, \ldots, k$. Note that $d_\ell^y \leq d_{\ell+1}^y$ for each $\ell$.

It follows from \eqref{composta-Sk} and Lemma \ref{lem00}\ref{b1} that
\begin{align*}
	{\|S_\lambda^k(y)\|}_X &\leq \prod_{\ell=1}^k F_\lambda(d_\ell^y) \cdot {\|y\|}_X \leq \left[ F_\lambda(d_k^y)\right]^k \cdot {\|y\|}_X,
\end{align*}
where the last inequality holds because $F_\lambda$ is an increasing function. This proves \ref{b2}.

Next, we establish 
	\begin{align} \label{eq.4}
		(T_\lambda \circ S_\lambda)(y) = y \ \ \mbox{for all} \ \ y \in X_\infty.
	\end{align}
If $y = 0$, both sides of the equation are equal to $0$. For $0 \neq y  \in X_\infty$, utilizing Lemma \ref{lem00}\ref{a1}, we have
	$$(T_\lambda \circ S_{\lambda,\mathfrak{q}_y} )(y) = y,$$
and the statement follows from $(\ref{def.of.Sz})$.

Finally, employing $(\ref{eq.4})$, we can conclude by induction that $$(T_\lambda^k \circ S_\lambda^k)(y) = y$$ for all $y \in X_\infty$ and $k \in \mathbb{N}$. This completes the proof of \ref{a2}.
\end{proof}

\begin{lemma} \label{lem0}
	Suppose $T$ as $(\ref{eq.1})$, $(\lambda^m)_{m=1}^\infty \subset \ell_1$ and $U \in \mathcal{L}(X) \setminus \{0\}$ such that
	\begin{align} \label{hip-lem0}
		\|T_{\lambda^m} - U\|_{\mathcal{L}(X)} \xrightarrow{m \to \infty} 0.
	\end{align} Then, there exists $k_0 \in \mathbb N$ such that, by denoting $\lambda^m =: (\lambda_k^m)_{k=1}^\infty$, we have:
	\begin{enumerate}[label=$(\mathsf{A3})$]
		\item \label{a3} $(\lambda_{k}^m)_{m=1}^\infty$ converges for each $k \in \mathbb N$.
	\end{enumerate}
	\begin{enumerate}[label=$(\mathsf{B3})$]
	\item \label{b3} $\lim\limits_{m \to \infty} |\lambda_{k_0}^m| > 0 = \lim\limits_{m \to \infty} \lambda_k^m$ if $k < k_0$.
\end{enumerate}
	\begin{enumerate}[label=$(\mathsf{C3})$]
		\item \label{c3}$ \displaystyle\left\|\sum_{k=k_0}^\infty \lambda_k^{m} \, B_{\bf w}^k - U\right\|_{\mathcal{L}(X)} \xrightarrow{m \to \infty} 0$.
	\end{enumerate}
\end{lemma}
\begin{proof}
	
For notational simplicity, we consider $\lambda_{k-\ell+1}^m \cdot \prod_{i=\ell}^k w_i= 0$ when $\ell > k$ in the subsequent analysis. For any $k, m, \ell  \in \mathbb N$, we have
	\begin{align*}
	\left|\lambda_{k-\ell+1}^m \cdot \prod_{i=\ell}^k w_i -  x_\ell^*(U(x_{k+1}))\right| &\stackrel{\eqref{eq-2.2}}{=}  \, {\left\| x_\ell^*\left(\sum_{n=1}^\infty \lambda_{k-n+1}^m \prod_{i=n}^k w_i \, x_n - U(x_{k+1})\right) \right\|}_X \\ &\stackrel{\eqref{eq-2.1}}{\leq} C_X \, {\left\| \sum_{n=1}^\infty \lambda_{k-n+1}^m \prod_{i=n}^k w_i \, x_n - U(x_{k+1})\right\|}_X  \\ &= C_X \, {\left\|T_{\lambda^m}(x_{k+1}) - U(x_{k+1})\right\|}_X \\ &\leq C_X \, {\|T_{\lambda^m} - U\|}_{\mathcal{L}(X)} \xrightarrow{m \to \infty} 0,
		\end{align*}
	which implies
	\begin{align} \label{lambda-limit}
		\begin{split}
		\lim_{m \to \infty} \lambda_{k-\ell+1}^m &= x_\ell^*(U(x_{k+1})) \, \prod_{i=\ell}^k w_i^{-1} \, \mbox{ if } \, 1 \leq \ell \leq k, \ \ \mbox{and} \\ x_\ell^*(U(x_{k+1})) &= 0 \ \ \mbox{if} \ \ \ell > k.
		\end{split}
	\end{align}
In particular, this proves \ref{a3}.

Suppose, for the sake of contradiction, that there does not exist a natural number $k_0$ satisfying condition \ref{b3}. In other words, by using \ref{a3}, we have
$$\lim_{m \to \infty} |\lambda_k^m| = 0 \ \ \mbox{for each} \ \ k \in \mathbb N.$$
From this, together with \eqref{lambda-limit}, we conclude that $x_\ell^*(U(x_{k+1}))=0$ for all $k,\ell \in \mathbb{N}$. Thus, according to \eqref{eq-2.4}, $U(x_k) = 0$ for each $k \geq 2$. On the other hand, it is evident that $U(x_1) = 0$ due to both \eqref{hip-lem0} and $T_{\lambda^m}(x_1) = 0$ whenever $m \in \mathbb{N}$. Therefore, we have just proven that $U(x_k) = 0$ for each $k \in \mathbb{N}$. This leads to a contradiction since $(x_k)_{k=1}^\infty$ is dense in $X$ (cf. \eqref{eq-2.3}) and $U \in \mathcal{L}(X) \setminus \{0\}$. Thus, there is $k_0 \in \mathbb N$ satisfying \ref{b3}.

We can prove \ref{c3} as follows:
	\begin{align*}
		{\left\|\sum_{k=k_0}^\infty \lambda_k^{m} \, B_{\bf w}^k - U \right\|}_{\mathcal{L}(X)} &\leq  {\left\|\sum_{k=1}^\infty \lambda_k^{m} \, B_{\bf w}^k - U\right\|}_{\mathcal{L}(X)} + {\left\|\sum_{k<k_0} \lambda_k^{m} \, B_{\bf w}^k\right\|}_{\mathcal{L}(X)} \\ &\leq \|T_{\lambda^{m}} - U\|_{\mathcal{L}(X)} + \sum_{k < k_0} |\lambda_k^{m}| \xrightarrow{m \to \infty} 0,
	\end{align*}
	where the convergence is due to $(\ref{hip-lem0})$ and \ref{b3}.
\end{proof}

To establish the supercyclicity of an operator $U \in \mathcal{L}(X)$, it is sufficient to demonstrate that $U$ satisfies the Supercyclicity Criterion, as defined by Salas \cite{Mont-rodrSal} and later simplified by Montes-Rodríguez and Salas \cite{Sal99} (refer to \cite[Definition 1.13 and Theorem 1.14]{BayMath-book}). This criterion necessitates the existence of an increasing sequence of integers $(n_k)$, along with two dense sets $\mathcal{D}_1$ and $\mathcal{D}_2$ in the space $X$, and a sequence of maps $S_{n_k}:\mathcal{D}_2 \rightarrow X$ satisfying the following conditions:
\begin{eqnarray}
	&& {\|U^{n_k}(x_0)\|}_X {\|S_{n_k}(y_0)\|}_X \to 0 \, \mbox{ whenever } \, x_0 \in \mathcal{D}_1 \, \mbox{ and } \, y_0 \in \mathcal{D}_2; \label{cond1-super}\\
	&& U^{n_k} S_{n_k}(y_0) \to y_0 \, \mbox{ for all } \, y_0 \in \mathcal{D}_2.\label{cond2-super}
\end{eqnarray}

The Supercyclicity Criterion is employed in the proof of our latest lemma.

\begin{lemma} \label{lem2}
	Assuming $T$ as $(\ref{eq.1})$, we have that each non-null operator in $\overline{T(\ell_1)}^{\mathcal{L}(X)}$ is supercyclic.
\end{lemma}
\begin{proof}
	Take $U \in \mathcal{L}(X) \setminus \{0\}$ and $(\lambda^m)_{m=1}^\infty \subset \ell_1$ such that 
	\begin{align*}
		{\|T_{\lambda^m} - U\|}_{\mathcal{L}(X)} \to 0 \ \ \mbox{as} \ \ m \to \infty.
	\end{align*}
We define $\lambda^m = (\lambda^m_k)_{k = 1}^\infty$ for each $m \in \mathbb N$. By applying Lemma \ref{lem0}, we can find a number $k_0 \in \mathbb N$ such that conditions \ref{a3}, \ref{b3}, and \ref{c3} hold. Note that all previously mentioned properties persist when replacing $(T_{\lambda^m})$ with $(T_{\lambda^m})_{m \geq m_0}$ for any arbitrary choice of $m_0 \in \mathbb{N}$. As a result, given that $\lim\limits_{m \to \infty} |\lambda_{k_0}^{m}| > 0$ according to \ref{b3}, it is without loss of generality to suppose that
\begin{align} \label{lem4-eqdelta}
	\inf_{m \in \mathbb N} |\lambda_{k_0}^{m}| =: \delta >0.
\end{align}

By setting $$R_m := \sum_{k=k_0}^\infty \lambda_k^{m} \, B_{\bf w}^k, \ \ \mbox{for each} \ \ m \in \mathbb N,$$
	it follows from \ref{c3} that
	\begin{align} \label{lem4-eq1}
		{\|R_m - U\|}_{\mathcal{L}(X)} \xrightarrow{m \to \infty} 0.
	\end{align}

Note that for $m \not= n$ and $x \in X$, we have
	\begin{align}  \label{lem4-eq2}
		\begin{split}
			(R_m \circ R_n)(x) &= R_m\left(\sum_{k=k_0}^\infty \lambda_k^{n} \, B_{\bf w}^k(x)\right) = \sum_{k=k_0}^\infty \lambda_k^{n} \, R_m(B_{\bf w}^k(x)) \\ &= \sum_{k=k_0}^\infty \lambda_k^{n} \, \sum_{r=k_0}^\infty \lambda_r^{m} \, (B_{\bf w}^r \circ B_{\bf w}^k)(x) = \sum_{k,r = k_0}^\infty \lambda_k^{n} \, \lambda_r^{m} \, B_{\bf w}^{r+k}(x),
		\end{split}
	\end{align}
	where the last equality is obtained because both $\lambda^m$ and $\lambda^n$ are included in $\ell_1$, ensuring the absolute convergence of the series and, consequently, the commutation of the terms.

	From $(\ref{lem4-eq2})$, we can deduce that $R_m \circ R_n = R_n \circ R_m$ for all $n, m \in \mathbb{N}$. Combining this with $(\ref{lem4-eq1})$, we obtain
	$$R_m \circ U = \lim_{n \to \infty} R_m \circ R_n = \lim_{n \to \infty} R_n \circ R_m = U \circ R_m \ \ (\mbox{in } \mathcal{L}(X))$$
	for every $m \in \mathbb{N}$. Hence, $U$ commutes with every operator $R_m$, where $m \in \mathbb{N}$. Consequently, for every $k,m \in \mathbb{N}$, we have
	\begin{align} \label{lem4-eq3}
		\begin{split}
			\|R_m^k - U^k\|_{\mathcal{L}(X)} &= \left\|(R_m - U)\left(\sum_{i=0}^{k-1} R_m^{k-i-1} \, U^{i}\right)\right\|_{\mathcal{L}(X)} \\ &\leq \|R_m - U\|_{\mathcal{L}(X)} \, \sum_{i=0}^{k-1} \|R_m\|^{k-i-1}_{\mathcal{L}(X)} \|U\|^i_{\mathcal{L}(X)} \xrightarrow{m \to \infty} 0.
		\end{split}
	\end{align}

Given $$\widetilde{\lambda}^m = (\widetilde{\lambda}_k^m)_{k=1}^\infty := (0,\ldots,0,\lambda_{k_0}^m, \lambda_{k_0+1}^m, \ldots), \ \ m \in \mathbb N,$$ and utilizing Lemma \ref{lem1}, we deduce that for each $T_{\widetilde{\lambda}^m} = R_m$, there exists a map $\widetilde{S}_m : X_\infty \to X$ with the following properties for any $y \in X_\infty$ and $k \in \mathbb{N}$:
	\begin{align} \label{lem4-eq4}
		\widetilde{S}_m(X_\infty) &\subset X_\infty, \ \ (R_m^k \circ \widetilde{S}_m^k)(y) = y, \ \ \mbox{and}
	\end{align} 
\begin{align} \label{lem4-eq7} 
	\begin{split}
	{\|\widetilde{S}^k_m(y)\|}_X &\leq \left[F_{\widetilde{\lambda}^m}(d_k^y)\right]^k \cdot {\|y\|}_X \stackrel{\eqref{lem4-eqdelta}}{\leq} \left[G_{k_0,\delta}(d_k^y)\right]^k \cdot {\|y\|}_X \mbox{ if } y \not=0,
	\end{split}
\end{align}
where $d_k^y = (k-1) \, k_0 + \mathfrak{q}_{y}$ and $G_{k_0,\delta}$ is defined as in \eqref{def-G} using the family of sequences $(\widetilde{\lambda}_k^m)_{k=1}^\infty, m \in \mathbb N$. Note that the validity of \ref{a3} ensures that $G_{k_0, \delta}$ is well-defined. Additionally, it is noteworthy that the upper bound of ${\|\widetilde{S}^k_m(y)\|}_X$ is independent of $m$ due to the utilization of \eqref{lem4-eqdelta}.

From $(\ref{lem4-eq3})$, we may choose a sequence $m_1 < m_2 < \cdots < m_k < \cdots$ such that
	\begin{align} \label{lem4-eq4half}
		{\|R_{m_k}^k - U^k\|}_{\mathcal{L}(X)} \leq 2^{-k} \, [G_{k_0,\delta}(k^2)]^{-k} \ \ \mbox{for each } \, k \in \mathbb N.
	\end{align}
	For each $k \in \mathbb N$, we define $S_k := \widetilde{S}_{m_k}^k : X_\infty \to X$. The well-defined nature of this mapping arises from the inclusion $\widetilde{S}_{m_k}(X_\infty) \subset X_\infty$ (cf. \eqref{lem4-eq4}). 
	
At this point, it is noteworthy that the definition of the maps $S_k$ solely depends on $U$ (and not on $y$). Specifically, its definition essentially relies on the choices of $k_0$, $\delta$, and $\widetilde{S}_m$, which, in turn, are determined by the selection of the sequence $T_{\lambda^m}$ converging to $U$.

Our aim is to establish that $U$ satisfies the Supercyclicity Criterion (cf. \eqref{cond1-super} and \eqref{cond2-super}), with the selection of $(n_k) = (k)$, $S_k$ as described above, and $\mathcal{D}_1 = \mathcal{D}_2 = X_\infty$.  Let $x_0$ and $y_0$ be two non-null vectors in $X_\infty$. By considering the definition of $R_m$ (which involves infinite sums of compositions of ``weighted backward shift operators''), we can readily observe the existence of $\widetilde{k}_1 \in \mathbb N$ such that ${\|R_{m_k}^k(x_0)\|}_X = 0$ for all $k \geq \widetilde{k}_1$, since $x_0$ has finite support. Consequently, for all $k \geq \widetilde{k}_1$, the following calculation holds:
	\begin{align} \label{lem4-eq5}
		\begin{split}
		{\|U^k(x_0)\|}_X &\leq {\|R_{m_k}^k(x_0) - U^k(x_0)\|}_X + {\|R_{m_k}^k(x_0)\|}_X \\ &\stackrel{\eqref{lem4-eq4half}}{\leq} 2^{-k}[G_{k_0,\delta}(k^2)]^{-k} \, {\|x_0\|}_X.
		\end{split}
	\end{align}
	On the other hand, it follows from $(\ref{lem4-eq7})$ that
	\begin{align} \label{lem4-eq6}
		{\|S_k(y_0)\|}_X = {\|\widetilde{S}_{m_k}^k(y_0)\|}_X \leq [G_{k_0,\delta}(d_k^{y_0})]^k \, {\|y_0\|}_X.
	\end{align}
	Since $G_{k_0,\delta}$ is an increasing function and $d_k^{y_0} \cdot k^{-2} \to 0$ as $k \to \infty$, it is clear that condition (\ref{cond1-super}) in the Supercyclicity Criterion is satisfied, by utilizing $(\ref{lem4-eq5})$ and $(\ref{lem4-eq6})$.

	Furthermore, we can demonstrate the validity of condition (\ref{cond2-super}). This can be established through the subsequent calculation:
	\begin{align*}
		\begin{split}
			{\|U^k \, S_k (y_0) - y_0\|}_{X} &\stackrel{(\ref{lem4-eq4})}{=} {\|U^k \, \widetilde{S}_{m_k}^k(y_0) - R_{m_k}^k \, \widetilde{S}_{m_k}^k (y_0)\|}_{X} \\ &\leq {\|U^k - R_{m_k}^k\|}_{\mathcal{L}(X)} \, {\|\widetilde{S}_{m_k}^k(y_0)\|}_X \\ &\hspace{-0.25in}\stackrel{(\ref{lem4-eq7}) + (\ref{lem4-eq4half})}{\leq} 2^{-k}[G_{k_0,\delta}(k^2)]^{-k} [G_{k_0,\delta}(d_k^{y_0})]^k \, {\|y_0\|}_X \xrightarrow{k \to \infty} 0.
		\end{split}
	\end{align*}
	That completes the proof.
\end{proof}

\bibliographystyle{amsplain}

\end{document}